\documentclass[12pt]{article}
\usepackage{epsfig,amsmath,amsthm,amssymb,graphics,amsfonts,psfrag}
\usepackage[letterpaper,margin=1in]{geometry}

\theoremstyle{plain}
\newtheorem{thm}{Theorem}[section]

\newtheorem{prop}[thm]{Proposition}

\theoremstyle{definition}

\newtheorem{ex}[thm]{Example}

\def\R{\mathbb{R}}

\def\N{\mathbb{N}}
\def\I{\infty}

\newcommand{\be}{\begin{equation}}
\newcommand{\ee}{\end{equation}}
\newcommand{\bea}{\begin{eqnarray}}
\newcommand{\eea}{\end{eqnarray}}
\newcommand{\beann}{\begin{eqnarray*}}
\newcommand{\eeann}{\end{eqnarray*}}
\newcommand{\benn}{\begin{equation*}}
\newcommand{\eenn}{\end{equation*}}

\def\ra{\rightarrow}
\def\I{\infty}

\newcommand{\cA}{{\mathcal A}}  
\newcommand{\cB}{{\mathcal B}}  
\newcommand{\cD}{{\mathcal D}}  
\newcommand{\cE}{{\mathcal E}}  

\begin{document}
 
\title{Dynamical analysis of evolution equations\\ in generalized models}
\author{Christian Kuehn\thanks{Max Planck Institute for Physics of Complex Systems, 01187 Dresden, Germany \& Center for Dynamics Dresden [CfD]},~ Stefan Siegmund\thanks{Department of Mathematics, TU Dresden, 01062 Dresden, Germany \& Center for Dynamics Dresden [CfD]}~ and Thilo Gross\thanks{Max Planck Institute for Physics of Complex Systems, 01187 Dresden, Germany \& Center for Dynamics Dresden [CfD]}}

\maketitle

\begin{abstract}
Generalized models provide a framework for the study of evolution equations without specifying all functional forms. The generalized formulation of problems has been shown to facilitate the analytical investigation of local dynamics and has been used successfully to answer applied questions. Yet their potential to facilitate analytical computations has not been realized in the mathematical literature. In the present paper we introduce the method of generalized modeling in mathematical terms, supporting the key steps of the procedure by rigorous proofs. Further, we point out open questions that are in the scope of present mathematical research and, if answered could greatly increase the predictive power of generalized models. 
\end{abstract}


\section{Introduction}

Many processes observed in nature are too complex to be described on a detailed mechanistic level. Therefore mathematical modeling can often not provide an exact set of evolution equations. This particularly evident in the context of mathematical biology \cite{Guckenheimer14}:\\

\textit{``The Hodgkin-Huxley models are based upon sound biophysical principles, but these principles do not constrain the models to a definite set of equations [$\ldots$].''}\\

Therefore we always seem to face the dilemma that a dynamical analysis requires a given specific model. However, once we specify some of the functions, that are only partially or not at all known, then we cannot provide a result that is valid for the underlying physical, chemical or biological process in full generality.\\ 

One possibility to address this problem is to introduce a wide variety of parameters or even phase-space variables in an ad-hoc way. A more systematic approach is provided by the theory of S-systems \cite{SavageauVoit,VoitSavageau} that aims at grouping different terms in evolution equations. Another systematic approach is considered in metabolic control theory \cite{KacserBurns,HeinrichRapoportRapoport,Reder} where a linearized analysis for a dynamical system uses sensitivities as a standard set of parameters.\\
  
Generalized modeling applies normalizing coordinate transformations to a system grouped into gain and loss terms to obtain a systematic parametrization. A generalized model provides an intermediate alternative enabling the mathematical modeler to use the partial information he has available but still provides enough flexibility to treat many different alternative models simultaneously. 

Sections \ref{sec:applications}-\ref{sec:predator-prey} form an extended introduction to readers not familiar with generalized modeling whereas we focus on new results in Sections \ref{sec:gm}-\ref{sec:beyond}. In Section \ref{sec:applications} we review several results obtained in applications to show what conclusions can be drawn from a generalized model. For illustrating the application of the method in practice we briefly consider an example of a planar predator-prey model in Section \ref{sec:predator-prey}. 

In Section \ref{sec:gm} we provide systematic treatment of generalized models for arbitrary ordinary differential equations (ODEs) in $\R^n$. We analyze the normalizing transformations for generalized models in detail, supporting this key step of the procedure by mathematical proofs. We also prove results on the number of parameters and discuss the role of positivity assumptions. This approach provides a standard scheme for the application of generalized models under minimal mathematical assumptions.  

In Section \ref{sec:bifurcations} we answer several mathematical questions that arise from the framework of generalized modeling in the context of results and methods from bifurcation theory. The focus of this analysis is on genericity and the structure of bifurcation diagrams. We explain via several instructive examples for non-degeneracy conditions what information can or cannot be inferred from a bifurcation analysis of a generalized model.

In Section \ref{sec:beyond} we show that generalized models can also be applied to homogeneous steady state dynamics of a wide variety of other evolution equations including delay, partial and stochastic differential equations. We also indicate some recent progress on extending generalized modeling to nonlocal dynamics for the case of periodic orbits.

\section{Results in Applications}
\label{sec:applications}
Generalized modeling was originally proposed in the context of community ecology \cite{GrossFeudel}. 
Only subsequently it was recognized as a general approach \cite{GrossFeudel1} and applied to a wide verity of applications.
A comprehensive list of publications is \cite{GrossFeudel,GrossFeudel1,GrossRudolfLevinDieckmann,BaurmannGrossFeudel,
StiefsvanVoornKooiFeudelGross,ZumsandeGross,StiefsGrossSteuerFeudel,
vanVoornStiefsGrossKooiFeudelKooijman,GrossEbenhoehFeudel,GrossEbenhoehFeudel1,StiefsVenturinoFeudel,
GrossBaurmannFeudelBlasius,HoefenerSethiaGross,Steueretal,SteuerGrossSelbigBlasius,
YeakelStiefsNovakGross,ReznikSegre,AufderheideRudolfGross,GehrmannDrossel,ZumsandeStiefsSiegmundGross}.
In these applications generalized models for instance revealed essential factors for the stability of food webs
\cite{GrossRudolfLevinDieckmann}; 
resolved a discrepancy between different modeling approaches in ecology \cite{StiefsvanVoornKooiFeudelGross};
implicated a Hopf bifurcation as a cause of Paget's disease in humans \cite{ZumsandeStiefsSiegmundGross};
and provided insights in the stability of mitochondrial metabolism \cite{Steueretal}.

Besides the ability of generalized modeling to deal with unspecified relationships, the success of the approach builds
mainly on making the Jacobian matrix analytically accessible. 
As shown in more detail below, the Jacobians obtained form generalized models are given in explicitly and typically contain only simple functions of the parameters. 
In systems of small and intermediate size the bifurcations of a generalized model can thus often be computed explicitly by hand. 

For exploring larger generalized models one typically randomly samples the local stability in random points of the parameter space to build up a database that is subsequently explored by machine learning techniques. 
Because of the direct accessibility of the Jacobian matrix the evaluation of every sample point involves only the computation of the leading eigenvalue of a matrix. 
Because of the numerical efficiency of this computation the exploration of large parameter spaces becomes feasible.
For instance in \cite{GrossRudolfLevinDieckmann} a system containing 50 dynamical variables and 
thousands of unknown parameters was explored by 100 billion ($10^{11}$) samples, which were obtained in reasonable numerical time.   
Generalized modeling is thus one of very few approaches that has reasonable hope of scaling for instance to whole-organism models 
in systems biology. 

\begin{figure}[htbp]
\centering
 \includegraphics[width=0.7\textwidth]{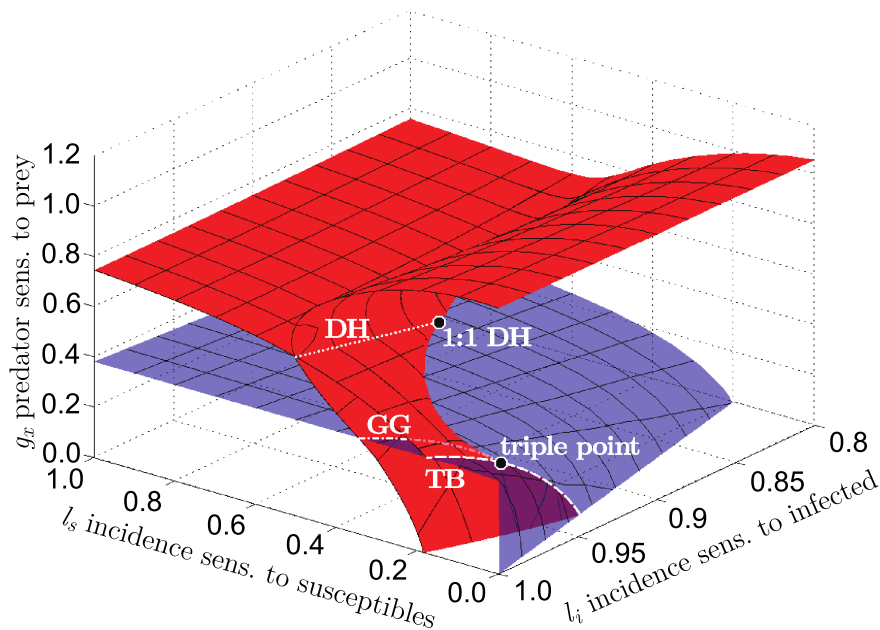}
\caption{\label{fig:Stiefs}Bifurcation diagram in generalized parameter space. Re-printed with permission from \cite{Stiefs}; see also \cite{StiefsVenturinoFeudel}. The red surface indicates Hopf bifurcation and the blue surface saddle-node bifurcations. Higher co-dimension bifurcations are indicated by separate labels. The main codimension two curves are Gavrilov-Guckenheimer (GG), Takens-Bogdanov (TB) and double-Hopf (DH) bifurcations.}
\end{figure}

A major limitation of generalized modeling -- the restriction to dynamics close to equilibria -- is so far mitigated in applications,
by the analysis of local bifurcations of higher codimension. These bifurcations allow some conclusions on global dynamical properties
and can be computed analytically in generalized models up to intermediate size and numerically beyond that.  
Figure \ref{fig:Stiefs} shows an example of the information that is obtained. The three-dimensional bifurcation diagram shows codimension-one bifurcation surfaces (fold and Hopf) as well as codimension-two curves (Gavrilov-Guckenheimer, Takens-Bogdanov and double-Hopf) and codimension three points ($1:1$ resonance, triple point \cite{Kuznetsov,Govaerts}). Although not all the unfoldings of the higher-codimension bifurcations are known it often suffices to detect the bifurcation point as an organizing center to gain insight into the overall dynamics. For example, parts of the double-Hopf bifurcation are known to generate chaotic dynamics due to associated torus and homoclinic bifurcations \cite{GH}.

\section{A Predator-Prey Example}
\label{sec:predator-prey}

To illustrate the basic steps in generalized modeling we apply it in a \emph{non-rigorous} way to a planar predator-prey system \cite{GrossFeudel1} with a prey density $X$ and a predator density $Y$. The prey population grows at a rate $S(X)$, predation occurs at rate $G(X,Y)$ and natural mortality of the predator at rate $M(Y)$ which yields
\be
\label{eq:gm_local}
\begin{array}{lcl}
X'&=& S(X)-G(X,Y),\\
Y'&=& G(X,Y)-M(Y).\\
\end{array}
\ee 
Suppose \eqref{eq:gm_local} admits an equilibrium point $(X,Y)=(X^*,Y^*)$ and introduce normalizing coordinates
\be
\label{eq:gm_local_normalize}
x:=\frac{X}{X^*}\qquad \text{and}\qquad y:=\frac{Y}{Y^*}.
\ee
moving the equilibrium to $(x,y)=(1,1)$. Then we normalize the rate functions
\be
\label{eq:gm_local_rates}
s(x):=\frac{S(X^*x)}{S(X^*)},\qquad g(x,y):=\frac{G(X^*x,Y^*,y)}{G(X^*,Y^*)}, \qquad m(y):=\frac{M(Y^*y)}{M(Y^*)}. 
\ee
Direct substitution of \eqref{eq:gm_local_normalize}-\eqref{eq:gm_local_rates} into \eqref{eq:gm_local} gives 
\be
\label{eq:gm_local1}
\begin{array}{lcl}
x'&=& \frac{S(X^*)}{X^*}s(x)-\frac{G(X^*,Y^*)}{X^*}g(x,y),\\
y'&=& \frac{G(X^*,Y^*)}{Y^*}g(x,y)-\frac{M(Y^*)}{Y^*}m(y),\\
\end{array}
\ee 
where we define new parameters
\be
\label{eq:gm_local_scale_ps}
\beta_s:=\frac{S(X^*)}{X^*},\quad \beta_1:=\frac{G(X^*,Y^*)}{X^*},\quad \beta_2:=\frac{G(X^*,Y^*)}{Y^*},\quad \beta_m:=\frac{M(Y^*)}{Y^*}.
\ee
Since $(x,y)=(1,1)$ is an equilibrium point we know that the following holds:
\be
\label{eq:gm_local_eq_cond}
\begin{array}{lclcl}
0&=&\frac{S(X^*)}{X^*}s(1)-\frac{G(X^*,Y^*)}{X^*}g(1,1)&=&\beta_s-\beta_1,\\
0&=&\frac{G(X^*,Y^*)}{Y^*}g(1,1)-\frac{M(Y^*)}{Y^*}m(1)&=&\beta_2-\beta_m.\\
\end{array}
\ee
Therefore \eqref{eq:gm_local1} can be re-written as
\be
\label{eq:gm_local2}
\begin{array}{lcl}
x'&=& \beta_1(s(x)-g(x,y)),\\
y'&=& \beta_2(g(x,y)-m(y)).\\
\end{array}
\ee 
The Jacobian at the equilibrium $(x,y)=(1,1)$ is then given by
\bea
\label{eq:local_Jac}
J(1,1)&=&\left(
\begin{array}{cc}
\beta_1~\partial_x [s(x)- g(x,y)]|_{(x,y)=(1,1)} & -\beta_1~\partial_y [g(x,y)]|_{(x,y)=(1,1)}\\ 
\beta_2~\partial_x [g(x,y)]|_{(x,y)=(1,1)} & \beta_2~\partial_y[g(x,y)- m(y)]|_{(x,y)={1,1}}\\ 
\end{array}
\right)\\
&=:&
\left(
\begin{array}{cc}
\beta_1[s_x-g_x] & -\beta_1g_y\\ 
\beta_2g_x & \beta_2[g_y-m_y]\\ 
\end{array}
\right)
\eea
where $\partial_x$, $\partial_y$ denote partial derivatives and we have introduced another set of parameters 
\be
\label{eq:gm_local_elasticities}
\begin{array}{lcl}
s_x=\partial_x(s(x))|_{x=1},&\quad &g_x=\partial_x(g(x,y))|_{(x,y)=(1,1)},\\
g_y=\partial_y(g(x,y))|_{(x,y)=(1,1)}, &\quad& m_y=\partial_y(m(y))|_{y=1}.\\ 
\end{array}
\ee
Based on this parametrization of the Jacobian, one can start to carry out a bifurcation analysis. From this example, we observe that the main steps of the method are:

\begin{enumerate}
 \item[(G1)] Build a model of the underlying process as evolution equations and group different terms in the resulting equations.
 \item[(G2)] Apply a transformation in phase space based on the existence of an equilibrium point and introduce parameters.
 \item[(G3)] Interpret the generalized parameters in the modeling context.
 \item[(G4)] Apply methods such as bifurcation analysis to characterize the dynamics of the generalized model.
\end{enumerate}

However, many open mathematical questions remain when we try to apply (G1)-(G4). It is the main goal of this paper to provide a detailed systematic and rigorous description of (G1)-(G4) for a wide variety of evolution equations. Furthermore, we will answer many questions that could not be answered appropriately in the context of applications; see references in Section \ref{sec:applications}.
 
\section{Structure of Generalized Models}
\label{sec:gm}

We start with generalized models for ordinary differential equations (ODEs). A general autonomous first-order system of ODEs is given by
\be
\label{eq:ODE}
\frac{dX}{dt}=X'=F(X;\mu)
\ee
where $X\in\R^n$ are phase space variables, $\mu\in\R^p$ are parameters and the vector field $F:\R^n\times \R^p\ra \R^n$ is assumed to be at least continuously differentiable in $X$ and continuous in $\mu$. If we can specify a particular map $F$ then the main task is to analyze the dynamics of \eqref{eq:ODE} i.e. to partition the parameter space $\R^p$ into regions of qualitatively equivalent dynamics \cite{Kuznetsov}. If we do not specify any assumptions on $F$ we focus on the abstract analysis of ODEs \cite{Hale}. Generalized models provide one possibility to bridge the gap between specific models and 
abstract analysis by making some structural assumptions on $F$ without specifying the map completely. We assume that \eqref{eq:ODE} has a 
\texttt{decomposition} of the form
\be
\label{eq:ODE1}
X_i'=F_i(X;\mu)=\sum_{k=1}^{K_i}a_{i,k} F_{i,k}(X;\mu), 
\ee
where $a_{i,k}=+1$ or $a_{i,k}=-1$, the subscript $i\in\{1,2,\ldots,n\}$ indicates the $i$-th coordinate, $K_i\geq 2$ and $F_i,F_{i,k}:\R^n\times \R^p\ra \R^+$. We are going to discuss the positivity assumption on the map $F$ at the end of this section. The terms $F_{i,k}$  with $a_{i,k}=1$ are called \texttt{gain terms} and those with $a_{i,k}=-1$ \texttt{loss terms}. We note that the type of decomposition is decided as part of the mathematical modeling and does not follow a fixed set of rules. However, the basic principle of grouping the different terms is often provided by their role in the mathematical model and the resulting system \eqref{eq:ODE1} has a systematic structure. 

A first step to understand the dynamics of \eqref{eq:ODE} is to analyze the stability and bifurcations of equilibria. Suppose there exists an equilibrium point $X^*$ so that $F(X^*;\mu)=0$. If we are only interested in the local dynamics near $X^*$ we can relax the differentiability assumptions on $F$ to a neighborhood of $X^*$. The local dynamics at $X^*$ is given to first-order by analyzing the eigenvalues of the Jacobian
\be
\label{eq:Jac}
\left.J(X;\mu)\right|_{X=X^*}=(D_XF)(X^*;\mu)=
\left(\frac{\partial F_i}{\partial X_j}(X^*;\mu)\right)_{ij}.
\ee
If we do not specify $F$ exactly then $X^*$ has to be treated as an unknown. The derivatives of functions/rates at the unknown equilibrium are often difficult to interpret in terms of physical parameters. Therefore, we would like to consider a transformation that allows a physical interpretation of parameters. Generalized modeling \cite{GrossFeudel} assumes that $X_i^*\neq 0$; we are going to discuss the special cases $X_i^*=0$ and $X_i^*\ra 0$ at the end of this section. Then one considers the \texttt{normalizing coordinate change}
\be
\label{eq:rescale}
x_i=\frac{X_i}{X^*_i}=:h_i(X), \qquad \text{for $i\in\{1,2,\ldots,n\}$.}
\ee
We remark that the idea of re-scaling to simplify or de-singularize a problem appears in several mathematical approaches for analyzing nonlinear systems; a typical example is provided by the blow-up method \cite{Dumortier,Dumortier1} that can be viewed as a phase space re-scaling in suitable coordinates. With the transformation \eqref{eq:rescale} the ODE \eqref{eq:ODE} transforms to 
\be
\label{eq:ODE_scaled}
x_i'=\frac{1}{X_i^*}F_i(X^*_1x_1,X^*_2x_2,\ldots,X^*_nx_n)=:\tilde{F}_i(x), 
\qquad \text{for $i\in\{1,2,\ldots,n\}$}
\ee
where we have omitted the $\mu$ parameter dependence for notational convenience. Therefore the equilibrium $X^*$ is transformed to $x^*=(1,1,\ldots,1)=:1$.\\

\textit{Remark:} A standard coordinate change in dynamical systems \cite{GH} is to consider the transformation $\bar{x}_i:=X_i-X_i^*$ so that the equilibrium point is moved to $\bar{x}=(0,\ldots,0)=:0$. This transformation is  mathematically convenient but does not provide a normalization of parameters as the coordinate change \eqref{eq:rescale}.\\

We can immediately check that the eigenvalues of the Jacobian \eqref{eq:Jac} remain unchanged.

\begin{prop}
\label{prop:det}
If $X^*_i\neq 0$ for all $i\in\{1,\ldots,N\}$ then the eigenvalues of the Jacobian are invariant under \eqref{eq:rescale} i.e. $\text{spec}(D_XF(X^*))=\text{spec}(D_x\tilde{F}(1))$.
\end{prop}

\begin{proof}
By direct calculation we find that 
\beann
\det\left[(D_x\tilde{F})(1)-\lambda ~\text{Id}\right]
&=&\det\left[\left.\left(\frac{\partial }{\partial x_j} \frac{1}{X_i^*}
F_i(X^*_1x_1,X^*_2x_2,\ldots,X^*_nx_n)\right)_{ij}\right|_{x=1}-\lambda ~\text{Id}\right]\\
&=& \det \left[\left(\frac{X^*_j}{X^*_i}\frac{\partial F_i}{\partial X_j}(X^*)
-\frac{X_j^*}{X_i^*}\lambda \delta_{ij}\right)_{ij}\right]\\
&\stackrel{(a)}{=}& \left(\prod_{i=1}^n \frac{1}{X^*_i}\right) 
\det \left[\left(X_j^*\frac{\partial F_i}{\partial X_j}(X^*)
-X_j^*\lambda \delta_{ij}\right)_{ij}\right]\\
&\stackrel{(b)}{=}& \left(\prod_{j=1}^n X^*_j\right) \left(\prod_{i=1}^n \frac{1}{X^*_i}\right) 
\det \left[\left(\frac{\partial F_i}{\partial X_j}(X^*)
-\lambda \delta_{ij}\right)_{ij}\right]\\
&=&\det[(D_XF)(X^*)-\lambda~ \text{Id}]
\eeann
where we have factored out non-zero scalars in step $(a)$ for each row and in $(b)$ for each column using linearity of determinants with respect to rows and columns. The result follows. 
\end{proof}

The equivalence of eigenvalues and the associated stability properties turns out to be of primary importance in many applications of generalized models \cite{GrossRudolfLevinDieckmann}; see also Section \ref{sec:applications}. However, Proposition \ref{prop:det} can also be viewed as a corollary to the following global result. 

\begin{prop}
\label{eq:equiv}
Suppose that $F\in C^k$ for some $k\in\N_0 \cup \{\omega\}$, where $C^\omega$ denotes analytic functions, and that $X^*_i\neq 0$. Then the ODEs \eqref{eq:ODE} and \eqref{eq:ODE_scaled} are $C^k$-smoothly equivalent via the map \eqref{eq:rescale}.
\end{prop}

\begin{proof}
Observe that $x=h(X)=(h_1(X),\ldots,h_n(X))$ is a $C^k$-diffeomorphism that conjugates 
the vector fields 
\benn
F(X)=(D_Xh)^{-1}(X)~\tilde{F}(h(X)).
\eenn
\end{proof}

The global smooth equivalence we showed is much stronger than (local) topological equivalence \cite{Kuznetsov} and the normalizing coordinate change \eqref{eq:rescale} can be viewed as leaving the dynamics completely unchanged. The next steps of generalized modeling involves grouping and labeling the free parameters so that they can be interpreted as modeling parameters. We introduce a notation for the normalized gain and loss terms
\be
\label{eq:func_normalized}
f_{i,k}(x):=\frac{F_{i,k}(X_1^*x_1,\ldots,X^*_nx_n)}{F_{i,k}(X^*)}=
\frac{F_{i,k}(X_1^*x_1,\ldots,X^*_nx_n)}{F_{i,k}^*}
\ee
where $F_{i,k}^*:=F_{i,k}(X^*)$ and we assume that $F_{i,k}(X^*)\neq 0$. Then \eqref{eq:ODE_scaled} reads
\be
\label{eq:ODE_scaled1}
x_i'=\sum_{k=1}^{K_i}a_{i,k}\frac{F_{i,k}^*}{X_i^*}f_{i,k}(x), 
\qquad \text{for $i\in\{1,2,\ldots,n\}$.}
\ee
As a next step we group two parameters together
\be
\label{eq:def_beta}
\tilde{\beta}_{i,k}:=\frac{F_{i,k}^*}{X_i^*}.
\ee

\textit{Remark:} From a physical point of view, we can also consider the units in the definition of $\tilde{\beta}_{i,k}$. $F_{i,k}^*$ is always a rate, for example mass per unit time. Since $X^*$ has the dimension of mass this implies that $\tilde{\beta}_{i,k}$ has the dimension 1/time.\\

Using definition \eqref{eq:def_beta} and the equilibrium point condition $x_i'=0$ at $x^*=1$ in \eqref{eq:ODE_scaled1} gives $n$ conditions
\be
\label{eq:lin_eq_eli}
0=\sum_{k=1}^{K_i}a_{i,k}\tilde{\beta}_{i,k} \qquad \text{for $i\in\{1,2,\ldots,n\}$.}
\ee
Therefore, we can hope to eliminate $n$ parameters. For example, we could try eliminating $\tilde{\beta}_{i,1}$ and set $\tilde{\beta}_{i,1}=a_{i,1}(-\sum_{a_{i,k}=1,k\neq1}\tilde{\beta}_{i,k}
+\sum_{a_{i,k}=-1}\tilde{\beta}_{i,k})$.
This elimination procedure can be formalized as follows: Define the vector 
\be
\label{eq:formal_beta}
\tilde{\beta}:=(\tilde{\beta}_{1,1},\tilde{\beta}_{1,2},\ldots,\tilde{\beta}_{1,K_1},\tilde{\beta}_{2,1},\ldots,\tilde{\beta}_{n,K_n})^T\in\R^\kappa
\ee
where $\kappa=\sum_{i=1}^n K_i$. Then \eqref{eq:lin_eq_eli} can be re-written as a matrix equation 
\be
\label{eq:lin_mat_el}
0=\mathcal{A}\tilde{\beta}
\ee
where the $n\times \kappa$ matrix $\mathcal{A}$ has elements in $\{-1,0,1\}$. The rank-nullity theorem gives
\be
\label{eq:ranknullity}
\kappa=\dim(\text{ker}(\cA))+\dim(\text{im}(\cA))=\dim(\text{ker}(\cA))+\text{rank}(\cA)).
\ee
This shows that $\text{rank}(\cA)$ is the number of parameters that we can eliminate and that $\dim(\text{ker}(\cA))$ is the number of remaining parameters after the linear relations \eqref{eq:lin_eq_eli} have been applied. The elimination of parameters is related to concepts used in structural kinetic modeling where the matrix $\cA$ is closely related to the stoichiometric reaction matrix with normalized entries \cite{Steueretal,HeinrichSchuster,SteuerGrossSelbigBlasius}.\\
 
Observe that we have a choice which parameters $\tilde{\beta}_{i,k}$ we eliminate using the algebraic equations \eqref{eq:lin_mat_el}. A further optional step is to introduce a parameter $\alpha_i$ for each variable and set
\be
\label{eq:beta}
\beta_{i,k}:=\frac{\tilde{\beta}_{i,k}}{\alpha_i}.
\ee 
where we assume that $\alpha_i>0$.\\

\textit{Remark:} From a physical point of view, we want to introduce $\alpha_i$ to \texttt{nondimensionalize}. This implies that $\alpha_i$ has to have the dimension 1/time while $\beta_{i,k}$ is dimensionless; in this case, we can interpret $\beta_{i,k}$ as ratios. One particular important choice to make this interpretation more precise is to consider the possible definition \cite{GrossFeudel}
\be
\label{eq:def_alpha_alt}
\alpha_i:=\sum_{k:~a_{i,k}=1} \tilde{\beta}_{i,k}=\sum_{k:~a_{i,k}=-1} \tilde{\beta}_{i,k}
\ee
where the equality between the two sums follows from \eqref{eq:lin_mat_el}. Using this definition we find that 
\benn
\beta_{i,k}=\frac{\tilde{\beta}_{i,k}}{\sum_{k:~a_{i,k}=1} \tilde{\beta}_{i,k}}=\frac{\tilde{\beta}_{i,k}}{\sum_{k:~a_{i,k}=-1} \tilde{\beta}_{i,k}}
\eenn
which interprets $\beta_{i,k}$ as the rate associated to the term with index $(i,k)$ divided by the total gain (or loss) rate i.e. we have obtained a ratio; see also Section \ref{sec:applications}. We shall not make explicit use of definition \eqref{eq:def_alpha_alt} here as it can be viewed as one particular choice of nondimensionalization. One can define a \texttt{time scale} as a physical quantity that has units 1/time. Therefore we shall call $\alpha_i$ \texttt{time scale parameters} from now on. Note that this justifies our assumption $\alpha_i>0$ on the basis of the underlying physical process.\\

Now we can re-write the differential equation \eqref{eq:ODE_scaled1} as 
\be
\label{eq:ODE_scaled2}
x_i'=\alpha_i\left(\sum_{k=1}^{K_i}a_{i,k}\beta_{i,k}f_{i,k}(x)\right), 
\qquad \text{for $i\in\{1,2,\ldots,n\}$}
\ee
where the relation \eqref{eq:lin_mat_el} is understood to apply as well. We call 
the parameters $\alpha_{i}$ and $\beta_{i,k}$ (resp. $\tilde{\beta}_{i,k}$) 
\texttt{scale parameters}. Obviously we have introduced quite a number of scale parameters to 
avoid specifying the functions in our model; therefore it is important to know how many scale 
parameters will appear in the model. We have the following result:

\begin{prop}
\label{prop:para_num}
The number of scale parameters for a generalized model \eqref{eq:ODE_scaled1} is as follows: 
\begin{enumerate}
 \item[(C1)] If $\tilde{\beta}_{i,k}$ are the only scale parameters and $\tilde{\beta}_{i,k}\neq \tilde{\beta}_{l,m}$ for all pairs $(i,k)\neq (l,m)$ then the minimum number of scale parameters is given by 
\be
\label{eq:beta_tilde}
\kappa-\text{rank}(\cA)=\dim(\text{ker}(\cA)).
\ee 
If all $\tilde{\beta}_{i,k}$ appear as multiplicative factors after the elimination via $A\tilde{\beta}=0$ then one more parameter can be eliminated.
 \item[(C2)] If $\beta_{i,k}$, $\alpha_i$ are the scale parameters and $\beta_{i,k}\neq \beta_{l,m}$ for all pairs $(i,k)\neq (l,m)$ then the minimum number of scale parameters is
\be
\label{eq:beta_tilde1}
\dim(\text{ker}(\cA))+n-\eta-1
\ee 
\end{enumerate}
where $\eta$ is the number of scale parameters $\beta_{i,k}$ that appear as multiplicative pre-factors after the elimination via \eqref{eq:lin_mat_el}.
\end{prop}

\textit{Remark:} If the conditions $\tilde{\beta}_{i,k}\neq \tilde{\beta}_{l,m}$ resp. $\beta_{i,k}\neq \beta_{l,m}$ are violated then further parameters can obviously be eliminated. However, a violation of this condition is not generic within the class of vector fields we consider here so we shall not consider this situation any further; for more on genericity see Section \ref{sec:bifurcations}.\\

\begin{proof}(of Proposition \ref{prop:para_num})
The previous discussion leading up to equation \eqref{eq:ranknullity} yields \eqref{eq:beta_tilde}. The second part of (C1) that allows the elimination of one further parameter will be clear once we proved (C2). For (C2), we have $\dim(\text{ker}(\cA))$ parameters $\beta_{i,k}$ and $n$ parameters $\alpha_i$ after using the linear relations $\cA\beta=0$. Assume without loss of generality that in the first $\eta$ coordinates, the parameters $\beta_{1,k},\ldots,\beta_{\eta,k}$ appear as multiplicative prefactors so that the ODEs are 
\be
\label{eq:tscale_proof}
\begin{array}{rcl}
x_1'&=&\alpha_1\beta_{1,1}\sum_{k=1}^{K_1}a_{1,k}f_{i,k}(x),\\
&\vdots& \\
x_{\eta}'&=&\alpha_\eta\beta_{\eta,1}\sum_{k=1}^{K_1}a_{\eta,k}f_{\eta,k}(x),\\
x_{\eta+1}'&=&\alpha_{\eta+1}\sum_{k=1}^{K_1}\beta_{\eta+1,k}a_{\eta+1,k}f_{\eta+1,k}(x),\\
&\vdots&\\
x_{n}'&=&\alpha_{n}\sum_{k=1}^{K_1}\beta_{n,k}a_{n,k}f_{n,k}(x),\\
\end{array}
\ee
Now we define new time scale parameters $\tilde{\alpha}_i:=\alpha_i\beta_{i,k}=\tilde{\beta}_{i,k}$ for $i\in\{1,2,\ldots,\eta\}$ and $\tilde{\alpha}_i=\alpha_i$ for $i\in \{\eta+1,\ldots,n\}$ which transforms \eqref{eq:tscale_proof} to
\be
\label{eq:tscale_proof1}
\begin{array}{rcl}
x_1'&=&\tilde{\alpha}_1\sum_{k=1}^{K_1}a_{1,k}f_{i,k}(x),\\
&\vdots& \\
x_{\eta}'&=&\tilde{\alpha}_\eta \sum_{k=1}^{K_1}a_{\eta,k}f_{\eta,k}(x),\\
x_{\eta+1}'&=&\tilde{\alpha}_{\eta+1}\sum_{k=1}^{K_1}\beta_{\eta+1,k}a_{\eta+1,k}f_{\eta+1,k}(x),\\
&\vdots&\\
x_{n}'&=&\tilde{\alpha}_{n}\sum_{k=1}^{K_1}\beta_{n,k}a_{n,k}f_{n,k}(x).\\
\end{array}
\ee
Therefore we have eliminated $\eta$ additional parameters. To eliminate one more parameter we can choose one time scale parameter $\tilde{\alpha}_i$, say without loss of generality $\tilde{\alpha}_1$, and apply a time re-scaling
\benn
t\mapsto t/\tilde{\alpha}_1.
\eenn
Defining new parameters $\tilde{\alpha}_i/\tilde{\alpha}_1$ yields the final result.
\end{proof}

We continue by interpreting the parameters $\beta_{i,k}$. They are gain and loss ratios for each term in the decomposition. To analyze the stability of the equilibrium $x^*=1$ we define \texttt{elasticities}
\be
\label{eq:exp_para}
f_{i,k,x_j}:=\frac{\partial f_{i,k}}{\partial x_j}(1)
=\left.\frac{\partial f_{i,k}}{\partial x_j}(x)\right|_{x=1}.
\ee
The elasticities are sometimes also called or \texttt{exponent parameters}. Then we find the Jacobian of \eqref{eq:ODE_scaled2} at $x=1$
\be
\label{eq:Jac_gen}
J(1)=\left(\begin{array}{cccc}
\alpha_1 & 0 & \cdots & 0 \\
0 & \alpha_2 & \cdots & 0 \\
\vdots & & \ddots & \vdots \\
0 & \cdots & & \alpha_n \\
\end{array}\right)
\left(\begin{array}{ccc}
\sum_{k=1}^{K_1} a_{1,k}\beta_{1,k} f_{1,k,x_1}  & \cdots & \sum_{k=1}^{K_1} a_{1,k}\beta_{1,k} f_{1,k,x_n} \\
\vdots &  \ddots & \vdots \\
\sum_{k=1}^{K_n} a_{n,k}\beta_{n,k} f_{n,k,x_1} & \cdots & \sum_{k=1}^{K_n} a_{n,k}\beta_{n,k} f_{n,k,x_n} \\
\end{array}\right).
\ee 
We also refer to the set of scale and exponent parameters as \texttt{generalized parameters}.

The key input to the generalized modeling process from applications is that we can often interpret the scale parameters and the elasticities for a given application. To explain the term elasticities we examine their definition more closely 
and observe that 
\benn
f_{i,k,x_j}=\frac{\partial f_{i,k}}{\partial x_j}(1)=
\frac{X^*_j}{F_{i,k}(X^*)}\frac{\partial F_{i,k}}{\partial X_j}(X^*)=
X^*_j\left(\frac{\partial}{\partial X_j}\ln F_{i,k}(X)\right)_{X=X^*}
\eenn
which interprets the exponent parameters as (scaled) logarithmic derivatives. Logarithmic derivatives are often called elasticities, particularly in the context of modeling economic problems \cite{Friedman} and in metabolic control theory \cite{Fell}. We can also view the exponent parameter $f_{i,k,x_j}$ as the \texttt{sensitivity} to variations of $f_{i,k}$ in the direction $x_j$ at the equilibrium point \cite{SmithSzidarovszkyKarnavasBahil}. If certain specific functional forms $F_{i,k}(X)$ are known from the modeling process we can get even more information (see \cite{Gross1}, p.49). We give a few examples using uni-variate functions $F_{i,k}(X)$ with $X\in \R$:

\begin{center}
\begin{tabular}{|l|l|}
\hline
$F_{i,k}(X)$ & $f_{i,k,x}$ \\
\hline 
$AX^q$ & $q$ \\
$A\exp(BX^q)$ & $qBX^*$ \\
$\frac{A}{B^q+X^q}$ & $-q\left(\frac{X^*}{B+X^*}\right)^q$\\
$\frac{AX^p}{B+X^q}$ & $\frac{(B+(X^*)^p) (B p+(p-q) (X^*)^q)}{(B+(X^*)^q)^2}$\\
\hline
\end{tabular}
\end{center}

Further possible dependencies of the exponent parameters are easily derived by direct differentiation of the given functional form.\\

We also have to make sure that we can find generalized parameters that belong to at least one specific model. Here, we support this important condition for sampling analysis \cite{GrossRudolfLevinDieckmann}.

\begin{prop}
Suppose we are given scale parameters that satisfy \eqref{eq:lin_eq_eli}. Given any set of elasticities $f_{i,k,x_j}$ there exists a specific model for the given generalized parameters.
\end{prop}

\begin{proof}
The proof is constructive. We are going to define functions $F_{i,k}(X)$ that will produce the given set of generalized parameters. Fix some $i\in\{1,2,\ldots,n\}$ and $k\in\{1,\ldots,K_i\}$. Then define 
\benn
F_{i,k}(X):=\gamma_{i,k} X_1^{q_{i,k,1}}X_2^{q_{i,k,2}}\cdots X_n^{q_{i,k,n}}
\eenn
where we can assume without loss of generality that $\gamma_{i,k}\neq 0$ and we can choose $q_{i,k,j}$ for $j\in\{1,2,\ldots,n\}$. It follows from \eqref{eq:exp_para} that
\benn
f_{i,k}(x)=\frac{\gamma_{i,k} (X_1^*)^{q_{i,k,1}}x_1^{q_{i,k,1}}(X_2^*)^{q_{i,k,1}}x_2^{q_{i,k,2}}\cdots (X_n^*)^{q_{i,k,1}}x_n^{q_{i,k,n}}}{\gamma_{i,k} (X_1^*)^{q_{i,k,1}}(X_2^*)^{q_{i,k,2}}\cdots (X_n^*)^{q_{i,k,n}}}=
 x_1^{q_{i,k,1}}x_2^{q_{i,k,2}}\cdots x_n^{q_{i,k,n}}.
\eenn
Calculating the elasticities yields
\benn
f_{i,k,x_j}=\frac{\partial f_{i,k}}{\partial x_j}(1)=q_{i,k,j}.
\eenn
Since we are free to choose $q_{i,k,j}$ we can match the prescribed elasticities. Note carefully that the previous calculation was independent on the choice of $\gamma_{i,k}$. Now define
\benn
\gamma_{i,k}:=\tilde{\beta}_{i,k}\frac{X_i^*}{(X_1^*)^{q_{i,k,1}}\cdots (X_n^*)^{q_{i,k,n}}}
\eenn 
and observe that 
\benn
\frac{F_{i,k}^*(X_1^*,\ldots,X_n^*)}{X_i^*}=\tilde{\beta}_{i,k}
\eenn
so that we also match the prescribed scale parameters.
\end{proof}

As a last step we re-visit the positivity assumptions on the maps $F$ and the assumption $X_i^*\neq0$ for the equilibrium $X^*=(X_1^*,\ldots,X_n^*)$. First, suppose that $F_{i,k}(X^*)=F^*_{i,k}=0$ which implies that the definition \eqref{eq:func_normalized} cannot be used to define the function $f_{i,k}$. In this case, we have to replace each term of the form $\beta_{i,k}f_{i,k,x_j}$ in the Jacobian \eqref{eq:Jac_gen} by the standard un-normalized term
\benn
[\beta_{i,k}f_{i,k,x_j}]:=\frac{1}{\alpha_i}\left(\frac{\partial}{\partial x_j}\frac{F_{i,k}(X_1^*x_1,\ldots,X_n^*x_n)}{X_i^*}\right)_{x=1}
=\frac{1}{\alpha_i}\left(\frac{\partial}{\partial X_j}F_{i,k}(X)\right)_{X=X^*}.
\eenn
However, suppose we know from the mathematical modeling that 
\benn
F_{i,k}(X)=\sum_{l=1}^n(X_l-X^*_l)^{q_l}
\eenn
is a power function for some $q=(q_1,\ldots,q_n)$ with $q_l\geq 1$ for all $l\in\{1,\ldots,n\}$. Then $[\beta_{i,k}f_{i,k,x_j}]=1/\alpha_i$ if $q_j=1$ and $[\beta_{i,k}f_{i,k,x_j}]=0$ if $q_j>1$. Hence we can discard the terms arising from $F_{i,k}$ in the Jacobian if we know that $F_{i,k}$ is locally super-linear and vanishes at the equilibrium. Note that the Jacobian is undefined for $0<q<1$ but that this case is also excluded by our assumption that $F$ is at least continuously differentiable near the equilibrium point. 

Next, we consider the situation when $X_i^*=0$ and assume that $F_{i,k}(X^*)\neq 0$ for all $i,k$. Then the normalizing re-scaling transformation \eqref{eq:rescale} is not well-defined. If we know from the mathematical modeling that $X_i^*=0$ regardless of the parameters $\mu\in\R^p$ in the underlying model then we can simply define 
\benn
x_i:=X_i
\eenn
and carry out the generalized analysis. If $X^*_i(\mu)\ra 0$ as $\mu\ra \mu^*$ for some $\mu^*\in\R^p$ then this situation can also be incorporated into the generalized analysis. Indeed, notice that
\benn
\tilde{\beta}_{i,k}=\frac{F^*_{i,k}}{X^*_i}\stackrel{\mu\ra \mu^*}{\longrightarrow} \infty \qquad \text{and} \qquad
f_{l,k,x_i}=\frac{X^*_i}{F_{l,k}(X^*)}\frac{\partial F_{l,k}}{\partial X_i}(X^*)\stackrel{\mu\ra \mu^*}{\longrightarrow} 0.
\eenn
Therefore, the cases of large scale parameters $\tilde{\beta}_{i,k}$ (or $\beta_{i,k}$) and small elasticities involving partial derivatives with respect to $x_i$ naturally incorporate the cases where $X^*_i$ tends to zero. 
 
\section{Bifurcations of Generalized Models}
\label{sec:bifurcations}

In the previous section we have focused on the general algebraic structure of generalized models and the normalizing transformations. The next step is to investigate bifurcations in generalized models. We start with a brief review of some basic terminology and results that we are going to need throughout our analysis.\\

Using the $n\times n$ Jacobian matrix \eqref{eq:Jac_gen} we have access to the eigenvalues and their multiplicities at the equilibrium point $x=1$ ($X=X^*$). One can calculate the eigenvalues $\lambda_i$ numerically using standard methods such as un-symmetric QR factorization \cite{GolubvanLoan}. If $\Re(\lambda_i)\neq 0$ for all $i\in\{1,2,\ldots,n \}$ the Hartman-Grobman Theorem \cite{Hartman} implies that the flow near $x=1$ is locally topologically equivalent to the flow of the 
linearized system. In particular, we get asymptotic stability if $\Re(\lambda_i)<0$ for all $i$. Therefore a necessary condition for bifurcation under parameter variation is that $\Re(\lambda_i)=0$ for one (or multiple) eigenvalues. We briefly recall how to define an unfolding of a generic bifurcating family \cite{Wiggins,ArnoldGeomODE} as we need this terminology throughout this section. Let $F\in C^r(\R^n,\R^n)$ be a vector field defining the ODE \eqref{eq:ODE}. The smoothness $r$ will not be of primary relevance for us and we always assume that $r$ is sufficiently large in the following, at least locally. Let $J^s_x(F)$ denote the $s$-jet 
\benn
(x,F(x),DF(x),D^2F(x),\ldots,D^sF(x))
\eenn
of $F$ at $x$ with $s\leq r$; denote the associated space of jets by $J^s(\R^n,\R^n)$. The $s$-jet extension $\hat{F}$ 
of $F$ is a map 
\benn
\hat{F}:\R^n\ra J^s(\R^n,\R^n), \qquad \hat{F}(x)=J^s_x(F)
\eenn
that maps a phase space point to the associated jet; observe that we can identify the jet space $J^s(\R^n,\R^n)$ with $\R^m$ for a suitable $m$. Let $x_0\in \R^n$ denote a phase space point and let $U$ be a neighborhood of $x_0$ and set $V:=F(U)$; we shall restrict to studying the local behavior of the vector 
field near $x_0$ from now on. Let $\cE\subset J^s(U,V)$ denote the codimension $n$ subset of those $s$-jets 
that have an equilibrium point $x_0$ in $U$ \cite{Wiggins} where codimension is defined as 
\benn
\text{codim}(\cE)=\text{dim}(J^s(U,V))-\text{dim}(\cE)=m-\text{dim}(\cE).
\eenn
Let $\cB\subset \cE \subset J^s(U,V)$ denote the codimension $n+1$ set 
of vector fields with a non-hyperbolic equilibrium point $x_0$; observe that a non-hyperbolic equilibrium point 
is defined by conditions on $DF(x_0)$. Consider an $s$-jet $J^s_{x}(F)$ with a non-hyperbolic equilibrium point at $x_0$. Then $J^s_{x}(F)$ lies in a set $\cD\subseteq \cB$ of codimension $b$ in $J^s(U,V)$ for $b\geq n+1$ and we define the codimension of the equilibrium point $x_0$ as $b-n$. To understand the dynamics near a bifurcation point in $\cD$ we need a parametrized family of vector fields $F(x;\mu)$ so that the associated family of $s$-jets is transverse to $\cD$. Recall that transversality of maps $G\in C^r(\R^n,\R^m)$ to a submanifold $M\subset \R^m$ at $G(x_0)$ is defined by the requirement
\benn
DG(x_0)T_{x_0}\R^n+T_{G(x_0)}M=T_{G(x_0)}\R^m.
\eenn

The following theorem is of fundamental importance to justify the next steps. To state the theorem we recall that a property that holds for a countable intersection of dense open sets (i.e. on a residual set) is called generic.

\begin{thm}[Thom's Transversality Theorem, \cite{LuSing}]
\label{thm:Thom}
Let $M$ be a submanifold of $J^s(U,V)$. The set of maps $F\in C^r(U,V)$ whose $s$-jet extensions are transversal to $M$ is a residual set in $C^r(U,V)$ (for some $r$ depending on $s$ and $n$).
\end{thm}

Hence if we can find a parametrized family $F(x;\mu)$ with $\mu\in\R^p$ which is transverse to $\cD$ then we have constructed a generic representative. However, so far we have not taken the dynamics completely into account. Consider a vector field $F^*$ with $F^*(x_0,\nu_0)=0$. We say that $F^*(x,\nu)$ is induced from $F(x,\mu)$ near $(x_0,\mu_0)$ if there is a continuous map $\phi$, defined near $\mu_0$ with $\phi(\nu_0)=\mu_0$, so that
\benn
F^*(x,\nu)=F(x,\phi(\nu)).
\eenn

A parametrized family $F(x,\mu)$ is called a universal unfolding near an equilibrium point $(x_0,\mu_0)$ if every other parametrized family of $C^r$ vector fields is equivalent to a family of vector fields induced by $F(x,\mu)$. In general, it is difficult to verify for many bifurcations with higher-dimensional parameter spaces that a transversal family also forms a universal unfolding.\\

\textit{Remark:} Although the results described so far give a framework for the classification of bifurcation points according to codimension there are a few subtle technical points regarding e.g. the applicability of Thom's Theorem \ref{thm:Thom} or possible re-parametrizations of time and coordinate changes \cite{Wiggins}.\\

The classification of local bifurcation according to codimension \cite{Kuznetsov,GH} yields the following classification up to codimension two:

\begin{itemize}
 \item $codim=1$: Fold or saddle-node (single zero eigenvalue), Hopf (pair of pure imaginary eigenvalues).
 \item $codim=2$: Bogdanov-Takens (double real zero eigenvalues), Gavrilov-Guckenheimer or fold-Hopf (single zero eigenvalue and pure imaginary pair of eigenvalues), Hopf-Hopf (two pairs of pure imaginary eigenvalues), cusp (single zero eigenvalue and degenerate normal form coefficient), Bautin or generalized Hopf (pair of pure imaginary eigenvalues and zero first Lyapunov coefficient).
\end{itemize}

The eigenvalues of \eqref{eq:Jac_gen} depend on the generalized parameters so that we can detect necessary conditions for all codimension one and two bifurcations except cusp and Bautin bifurcation that are of codimension two as they violate 
codimension one non-degeneracy conditions. We shall not aim at a complete discussion of bifurcation analysis of generalized models but point out some features via examples.

\begin{ex}
\label{ex:bif1}
Consider the following generalized toy model 
\be
\label{eq:toy1D}
X'=F_1(X)-F_2(X)+F_3(X)
\ee
for $X\in \R$. A possible specific model covered by \eqref{eq:toy1D} is 
\be
\label{eq:SN_toy}
X'=X^2-2X+1-A, \qquad \text{for $A\in(-1,1)$}
\ee
where $F_1(X)=X^2$, $F_2(X)=2X$ and $F_3(X)=1-A$. It is easily checked that \eqref{eq:SN_toy} has two equilibria at $X^*_\pm=1\pm\sqrt{A}$ for $A>0$ that undergo a non-degenerate fold bifurcation at $A=0$. The bifurcation diagram is shown in Figure \ref{fig:fold}(a). After the normalizing scaling transformation $X=X^*_\pm x$ we find that \eqref{eq:SN_toy} gives two ODEs 
\be
\label{eq:SN_toy1}
x'=X^*_\pm x^2-2x+\frac{1-A}{X^*_\pm}=:b_\pm(x,A), \qquad \text{for $A\in[0,1)$.}
\ee

\begin{figure}[htbp]
\centering
\psfrag{x}{$x$}
\psfrag{X}{$X$}
\psfrag{A}{$A$}
\psfrag{a}{(a)}
\psfrag{c}{(c)}
\psfrag{b+}{(b+)}
\psfrag{b-}{(b-)}
\psfrag{beta}{$\beta$}
 \includegraphics[width=1\textwidth]{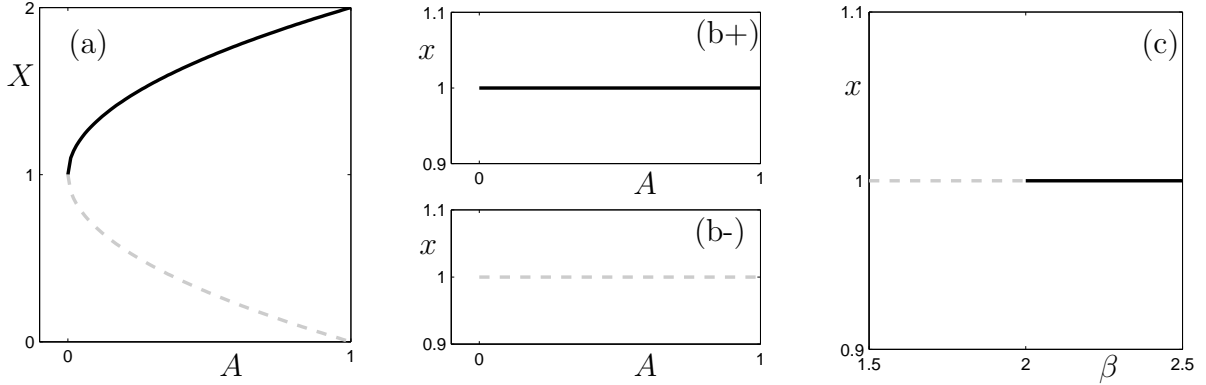}
\caption{\label{fig:fold}Comparison of conventional and generalized models in an example. Solid black curves are stable equilibria and dashed gray curves unstable equilibria. (a) Bifurcation diagram for the specific model \eqref{eq:toy1D} showing a fold bifurcation. (b) Fold bifurcation for the re-scaled equations \eqref{eq:SN_toy1}; we plotted the two equilibrium curves at $x=1$ separately to indicate that they are associated to different branches in (a). (c) Bifurcation diagram for the generalized model \eqref{eq:toy1D}.}
\end{figure}

The new bifurcation diagram of \eqref{eq:SN_toy1} is given in Figure \ref{fig:fold}(b$\pm$). The fold bifurcation from Figure \ref{fig:fold}(a) can still be recognized in Figure \ref{fig:fold}(b$\pm$) at $A=0$ since
\benn
\frac{\partial b_\pm}{\partial x}(1,0)=0 \quad \text{and} \quad 
\frac{\partial^2 b_\pm}{\partial x^2}(1,0)=2.
\qquad 
\eenn
The transversality condition requires looking at the parameter derivative
\be
\label{eq:der_ne}
\frac{\partial b_\pm}{\partial A}(x,A)=\pm\frac{x^2-1}{2\sqrt{A}}
\ee
and we observe that the condition is not well-defined for $(x,A)=(1,0)$. The generalized formulation of \eqref{eq:toy1D} is
\be
\label{eq:toy1D_gen}
x'=\alpha(\beta_{1}f_1(x)-\beta_2 f_2(x)+(\beta_2-\beta_1)f_3(x)).
\ee
Using a time re-scaling $t\mapsto t/(\alpha \beta_1)$ and setting $\beta:=\beta_2/\beta_1$ we get that \eqref{eq:toy1D_gen} can be written as
\be
\label{eq:toy1D_gen1}
x'=f_1(x)-\beta f_2(x)+(\beta-1)f_3(x).
\ee
The associated Jacobian for \eqref{eq:toy1D_gen1} at $x=1$ is
\benn
J(1)=f_{1,x}-\beta f_{2,x}+(\beta-1)f_{3,x}.
\eenn
If we assume that the exponent parameters are $(f_{1,x}.f_{2,x},f_{3,x})=(2,1,0)$ as for our specific model then we find that 
\benn
J(1)=2-\beta
\eenn
which satisfies the necessary condition for a saddle-node bifurcation at $\beta=2$. Since \eqref{eq:toy1D_gen} always has an equilibrium point at $x=1$ we find a bifurcation diagram as shown in Figure \ref{fig:fold}(c).
\end{ex}

From Example \ref{ex:bif1} we conclude that a generalized model will be able to detect the main bifurcation conditions for eigenvalue crossings but that one has to be careful in interpreting the shape of bifurcation diagrams. Furthermore, a first-order analysis will not check whether non-degeneracy and transversality conditions are satisfied. This raises the issue of genericity for generalized models which we address in the next example.

\begin{ex}
\label{ex:bif2}
We continue with Example \ref{ex:bif1} and the generalized model
\be
\label{eq:toy1D_gen1_new}
x'=f_1(x)-\beta f_2(x)+(\beta-1)f_3(x)=:g(x,\beta).
\ee
We check the generalized model for genericity near a fold bifurcation. Consider the associated $1$-jet space $J_x^1(\R,\R)$ to $g$ with elements
\benn
\left(x,g(x,\beta),\frac{\partial g}{\partial x}(x,\beta)\right)
\eenn 
which is three-dimensional. The bifurcation set $\cB$ for a fold bifurcation is given by
\benn
\{(x,0,0)\in J_x^1(\R,\R)\}
\eenn
which has dimension $1$ and codimension $b=2$. Therefore we get, as expected, that the fold bifurcation has codimension $b-n=2-1=1$. The $1$-jet extension to $g:\R\ra\R$ is the map
\be
\label{eq:1jet_ext}
\hat{g}(x,\beta)=\left(x,g(x,\beta),\frac{\partial g}{\partial x}(x,\beta)\right)
\ee
which we view as a map from $\hat{g}:\R^2\ra \R^3$ including the generalized parameter $\beta$. The linearization of the $1$-jet extension is
\be
\label{eq:ext_lin}
D\hat{g}(x,\beta)=\left(\begin{array}{cc} 
1 & 0 \\ \frac{\partial g}{\partial x}(x,\beta) & \frac{\partial g}{\partial \beta}(x,\beta) \\
\frac{\partial^2 g}{\partial x^2}(x,\beta) & \frac{\partial^2 g}{\partial x \partial \beta}(x,\beta) \\
\end{array}\right)
\ee
The tangent space to $\cB$ is spanned by the vector $(1,0,0)^T\in T_{\hat{g}(x,\beta)}\R^3$; note that we can obviously identify $T_{\hat{g}(x,\beta)}\R^3$ with $\R^3$. Evaluating \eqref{eq:ext_lin} at a fold bifurcation point $(1,\beta_0)$ we know that 
\benn
\frac{\partial g}{\partial x}(1,\beta_0)=0.
\eenn
This implies that checking transversality of the $1$-jet extension \eqref{eq:1jet_ext} reduces to checking that the $3\times 3$ matrix
\benn
H:=\left(
\begin{array}{ccc}
1 & 0 & 1\\ 0 & \frac{\partial g}{\partial \beta}(1,\beta_0) & 0 \\
\frac{\partial^2 g}{\partial x^2}(1,\beta_0) & \frac{\partial^2 g}{\partial x \partial \beta}(1,\beta_0) & 0\\
\end{array}
\right)
\eenn
is non-singular. This just means 
\benn
\det(H)=-\frac{\partial g}{\partial \beta}(1,\beta_0)\cdot \frac{\partial^2 g}{\partial x^2}(1,\beta_0) \neq 0.
\eenn 
This recovers the well-known \cite{Kuznetsov} non-degeneracy conditions 
\be
\label{eq:wellknown}
\frac{\partial g}{\partial \beta}(1,\beta_0)\neq 0 \qquad \text{and} \qquad \frac{\partial^2 g}{\partial x^2}(1,\beta_0)\neq 0.
\ee
Recall that if \eqref{eq:wellknown} holds then the unfolding is indeed universal in $\beta$ \cite{GH,Kuznetsov}. Only the last step of the genericity analysis has to be adapted to the generalized model. In particular, we can plug in the structure of the model \eqref{eq:toy1D_gen1_new} into \eqref{eq:wellknown} which yields the conditions
\be
\label{eq:gen_ndtv}
\begin{array}{lcl}
\frac{\partial g}{\partial \beta}(1,\beta_0)&=&-f_2(1)+f_3(1) \neq 0,\\
\frac{\partial^2 g}{\partial x^2}(1,\beta_0)&=&\frac{\partial^2 f_1}{\partial x^2}(1)-\beta_0 \frac{\partial^2 f_2}{\partial x^2}(1)+(\beta_0-1)\frac{\partial^2 f_3}{\partial x^2}(1) \neq 0.\\
\end{array}
\ee
The conditions \eqref{eq:gen_ndtv} restrict the function space for which one can expect that the generalized model has a fold bifurcation with a universal unfolding. The examples have also demonstrated that known results about bifurcations and unfoldings carry over easily to generalized models. The same calculations can be carried out for all other local bifurcations which will yield conditions analogous to \eqref{eq:gen_ndtv}.
\end{ex}

Even if the unfolding is not universal (e.g. if we had chosen a function space in Example \eqref{ex:bif2} with $f_2(1)=f_3(1)$) then we can often use the availability of additional parameters as the next example illustrates. 

\begin{ex}
Consider a planar generalized model 
\be
\label{eq:ex_generic}
\begin{array}{lcl}
x_1'&=&\sum_{k=1}^{K_1}a_{1,k}\beta_{1,k}f_{1,k}(x),\\ 
x_2'&=&\sum_{k=1}^{K_2}a_{2,k}\beta_{2,k}f_{2,k}(x),\\
\end{array}
\ee
where we assume that the time scales are all equal to $1$ for simplicity. Suppose \eqref{eq:ex_generic} has a Hopf bifurcation at $(x,\beta_{1,1})=(1,\beta_{1,1}^*)$ for all other generalized parameters fixed. The Jacobian $J(1)$ at $\beta_{1,1}=\beta_{1,1}^*$ has a complex conjugate pair of eigenvalues $\lambda_{1,2}=\lambda_{1,2}(\beta_{i,k},f_{i,k,x_j})$ depending on the generalized parameters with $c_1:=\text{Trace}(J(1))=0$. There are two non-degeneracy conditions \cite{Kuznetsov}. The transversality condition is
\be
\label{eq:gen_cond1}
c_2:=\left.\left(\frac{\partial}{\partial \beta_{1,1}} \Re(\lambda_{1,2})\right)\right|_{\beta_{1,1}=\beta_{1,1}^*}\neq 0.
\ee
The second non-degeneracy condition is that the first Lyapunov coefficient $l_1=:c_3$ is not equal to zero \cite{GH,KuehnCanLya}. The condition can be written in terms of the partial derivatives with respect to $x_{1,2}$ up to third order. In total, we have to satisfy three algebraic conditions depending upon the generalized parameters and the choice of functions $f_{i,k}$ to get a non-degenerate Hopf bifurcation 
\bea
c_1(\beta_{i,k},f_{i,k,x_j})&=&0,\label{eq:condHopf1}\\
c_2(\beta_{i,k},f_{i,k,x_j})&\neq& 0,\label{eq:condHopf2}\\
c_3(\beta_{i,k},f_{i,k})&\neq& 0.\label{eq:condHopf3}
\eea
Observe that the scale parameters $\beta_{1,k}$ or $\beta_{2,k}$ will appear as coefficients in the linear combination of $c_3$ so that varying one scale parameter we generically satisfy \eqref{eq:condHopf3} in the space of smooth functions $f_{i,k}\in C^r(\R^2,\R)$. Hence additional free generalized parameters can compensate for a restricted choice of functions $f_{i,k}$. Varying two more generalized parameters we can always satisfy \eqref{eq:condHopf1}-\eqref{eq:condHopf2} generically.
\end{ex}

We have seen that the scale parameters are relatively easy to understand. The elasticities are bit more complicated as the next example illustrates.

\begin{ex}
In principle, we can also just vary the elasticities and treat them as bifurcation parameters. The obvious caveat is that this might not yield a smooth family of functions everywhere in phase space. For example, if
\be
\label{eq:pow_ns}
F_{i,k}(X)=A(X-B)^p, \qquad \text{for $p>0$}
\ee
where $A,B$ are parameters, then $F_{i,k}$ is $C^\I$ for $X\neq B$ but only $C^k$ for some finite $k$ if $p\not\in \N$. This can again lead to violations of non-degeneracy conditions for bifurcations as in Example \eqref{ex:bif1}. In particular, as shown in \cite{KuehnScaleSN}, we will not be able to conclude quantitative universal scaling laws near bifurcations if we do not restrict the functions $F_{i,k}$ to be sufficiently smooth. Note that this phenomenon is again non-generic for a sufficiently large parameter space. The second key observation for elasticities relates to the 
form of $f_{i,k}$. For \eqref{eq:pow_ns} we find
\be
\label{eq:exp_p}
f_{i,k,x}=\frac{pX^*}{X^*-B}
\ee
which depends on the equilibrium point location $X^*$. The best way to interpret \eqref{eq:exp_p} is asymptotic knowledge about $X^*$; for example, we have 
\benn
\begin{array}{lrl}
X^*\gg 1, \quad B=O(1) & \Rightarrow &f_{i,k,x}\approx p,\\
X^*\approx B, \quad X^*=O(1) & \Rightarrow& f_{i,k,x}\gg 1,\\
0<X^*\ll 1, \quad B=O(1) &\Rightarrow & f_{i,k,x}\approx 0.\\
\end{array} 
\eenn 
where $O(1)$ indicates a fixed constant independent of any asymptotic limits of $X^*$.
\end{ex}

The last example shows that it can be beneficial to restrict the class of functions considered in a generalized model to a certain class, e.g. smooth functions, polynomials or certain functionals particular to the application area \cite{Gross1}. 

The previous three examples illustrate the main aspects of the genericity question for generalized models:

\begin{enumerate}
 \item We can check for a given decomposition and given function spaces whether the non-degeneracy conditions for 
a bifurcation are satisfied. 
 \item The checking depends crucially on the underlying mathematical modeling and how the functions and decomposition are chosen.
 \item Higher-order terms have to be taken into account that are not parametrized by scale parameters and elasticities; see also \cite{ZumsandeGross1}.
 \item Bifurcation and non-degeneracy conditions can usually be satisfied upon considering sufficiently large generalized parameter spaces. 
\end{enumerate}

When studying an uncertain system one can not necessary conclude that a bifurcation detected by generalized modeling is crossed transversally when a given (conventional) parameter is changed in the application. However, one can generically expect that bifurcations are transversal within the whole class of systems described by the generalized model.

Another issue that was already briefly considered in \eqref{ex:bif1} is the correspondence between generalized and specific models. If a specific system exhibits multistability one may ask how this is reflected in a generalized model.

\begin{ex}
\label{ex:bistable}
Consider the specific bi-stable system
\be
\label{eq:bistable}
X'=6-11X+6X^2-X^3
\ee 
which is easily seen to have three equilibria $X^*=1,2,3$ where $X^*=1,3$ are stable and $X^*=2$ is unstable. A generalized model encompassing \eqref{eq:bistable} as a special case is
\be
\label{eq:bistable_gen}
X'=-F_1(X)+F_2(X)-F_3(X).
\ee 
We can obtain the special case \eqref{eq:bistable} by $F_1(X)=11X-6$, $F_2(X)=6X^2$ and $F_3(X)=X^3$. Also the generalized parameter sets representing \eqref{eq:bistable} can easily be found
\be
\label{eq:res_bistable}
\begin{array}{lcl}
\text{if $X^*=1$} & \quad & (\beta_{1},\beta_2,\beta_3,f_{1,1,x},f_{1,2,x},f_{1,3,x})=(5,6,1,11/5,2,3) \\
\text{if $X^*=2$} & \quad & (\beta_{1},\beta_2,\beta_3,f_{1,1,x},f_{1,2,x},f_{1,3,x})=(8,12,4,11/8,2,3) \\
\text{if $X^*=3$} & \quad & (\beta_{1},\beta_2,\beta_3,f_{1,1,x},f_{1,2,x},f_{1,3,x})=(9,18,9,11/9,2,3) \\
\end{array}
\ee  
where we assumed without loss of generality for the single time scale that $\alpha=1$. The results \eqref{eq:res_bistable} illustrate that three different points in generalized parameter space represent the three local equilibria. Indeed, we have considered generalized modeling here as a purely local method (but see also Section \ref{ssec:nonlocal}) which also explains this representation of multistability.
\end{ex}

The previous example also raises a more general question about the correspondence between bifurcation results for specific and generalized models. In particular, one may want to relate the parameters $\mu\in\R^p$ in \eqref{eq:ODE} to the generalized parameters which can be expressed via a mapping
\be
\label{eq:mapM}
\{\mu\in\R^p\}\stackrel{M}{\longrightarrow}\{\alpha_i,\beta_{i,k},f_{i,k,x_j}\}_{i,j,k}.
\ee
Obviously this map $M$ depends on the specific functional forms of the vector field $F(x;\mu)$. Therefore we shall not discuss the parameter space mapping here. Figure \ref{fig:Stiefs2} illustrates the mapping $M$ for a model system from ecology \cite{StiefsvanVoornKooiFeudelGross}.\\  

\begin{figure}[htbp]
\centering
\psfrag{x}{$f_{1,1,x_1}$}
\psfrag{y}{$f_{2,1,x_1}$}
\psfrag{z}{$f_{1,2,x_1}$}
\psfrag{a}{(a)}
\psfrag{X2}{$X_2$}
\psfrag{b}{(b)}
\psfrag{mu}{$\mu_1$}
 \includegraphics[width=1\textwidth]{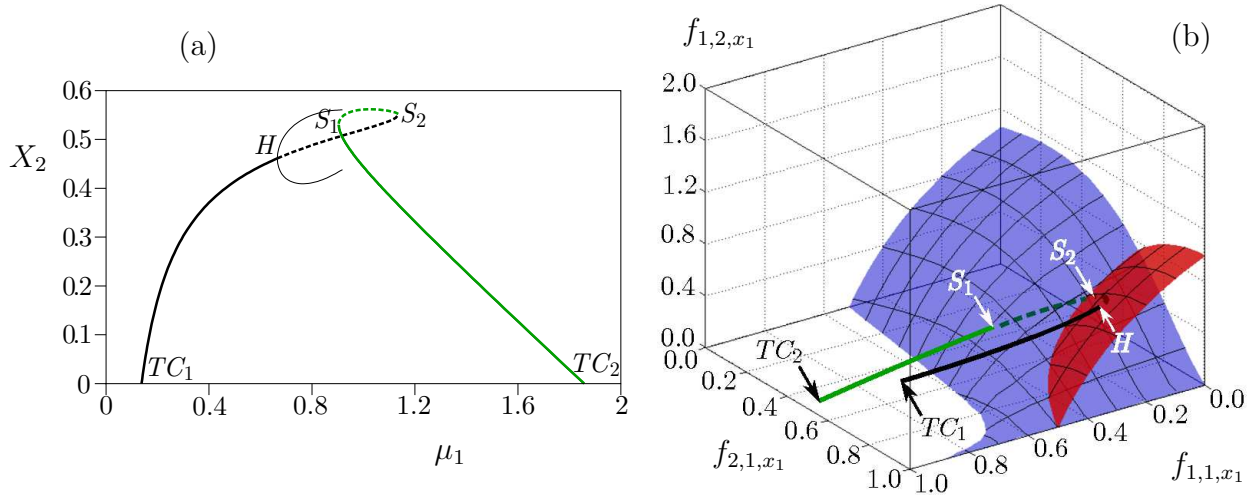}
\caption{\label{fig:Stiefs2}Example of the correspondence between specific and generalized bifurcation diagrams; figure adapted from \cite{StiefsvanVoornKooiFeudelGross}. (a) Bifurcation diagram for a specific model where $\mu_1$ is a parameter and $X_2$ a phase space variable. Thick lines denote equilibria with solid curves for stable and dashed curves for unstable equilibria. The thin lines mark the maximum and minimum $X_2$ coordinates of a limit cycle generated via the supercritical Hopf bifurcation $H$. Two transcritical bifurcation $TC_{1,2}$ and two saddle-node bifurcations $S_{1,2}$ are indicated as well. (b) Generalized parameter space bifurcation diagram where the two surfaces indicate Hopf (red) and saddle-node (blue) bifurcations. The curve represents the diagram from (a) obtained via a mapping \eqref{eq:mapM}. To calculate the curve one follows the same procedure as in Example \ref{ex:bistable}.}
\end{figure}

A very important conclusion of all previous examples is that generalized models produce information that ``scales with the input''. The more knowledge from the modeling process is available, the more detailed information about the dynamics can be derived.\\ 

We end this section by mentioning an important practical issue. We have to address how to find bifurcations and their associated varieties (such as surfaces, lines, curves, etc.) in parameter space. A direct method calculates the eigenvalues 
and uses iteration to find the zeros of real parts. An indirect method employs so-called test functions that vanish once a certain type of eigenvalue crossing occurs; see \cite{Govaerts} or \cite{Kuznetsov} for an overview of different test functions and references to the literature. We note that analytical methods such as the method of resultants \cite{GuckenheimerMyers1,GuckenheimerMyers2} have also been employed successfully in the context of generalized modeling \cite{GrossFeudel} using computer algebra \cite{GrossFeudel2}; having an explicit formula for the bifurcation loci can be beneficial for visualization \cite{StiefsGrossSteuerFeudel}. 

\section{Beyond ODEs}
\label{sec:beyond}

Although we presented the ideas of generalized modeling only in the context of ODEs so far, it is evident that their scope is much broader. Recall that the key ideas are:

\begin{itemize}
 \item There exist unknown functional forms in the model.
 \item Group the different parts of the vector field into gain and loss terms.
 \item Introduce a normalizing coordinate change for an equilibrium point.
 \item Re-scale the gain and loss terms and introduce scale parameters.
 \item Linearize at the new equilibrium $x=1$ and introduce elasticities.
 \item Interpret the generalized parameters and assign suitable ranges.
 \item Use tools from bifurcation analysis to capture the dynamics. 
\end{itemize}

These ideas carry over naturally into a much wider setup than just ODEs. Generalized models for discrete dynamical systems (iterated maps) are discussed in \cite{Karrasch}. Many other mathematical evolution equations are very similar to ODEs. In particular, the notion of equilibrium point, coordinate changes or re-scaling as well as linearized analysis carry over. In this section we shall demonstrate this observation for certain classes of partial, functional and stochastic differential equations in Sections \ref{ssec:PDE}-\ref{ssec:SDE}. In Section \ref{ssec:nonlocal} we provide a brief overview on extensions to nonlocal dynamics.

\subsection{Partial Differential Equations}
\label{ssec:PDE}

Generalized modelling for partial differential equations has first been developed for reaction-diffusion equations in the context of spatial predator-prey models \cite{GrossBaurmannFeudelBlasius,BaurmannGrossFeudel}. A more general class of partial differential equations (PDEs) that can be analyzed by dynamical systems techniques is given by (semilinear) parabolic equations. Let $(x,t)\in \Omega \times \R^+$ where $\Omega\subset \R^n$ is a domain and let $u=u(\cdot,t)\in X$ for all $t\geq 0$ where $X$ is a suitable Banach space e.g. $X=L^p(\Omega)$; for details see \cite{Henry}. Consider the abstract evolution equation
\be
\label{eq:PDE}
\frac{\partial u}{\partial t}+Au=F(x,t,u;\mu)
\ee 
where $F:\Omega \times \R^+ \times X\times \R^p \ra \R^n$ and $A$ is a sectorial (differential) operator on $X$ \cite{Henry}. A concrete example for \eqref{eq:PDE} are initial boundary value problems for reaction-diffusion PDEs \cite{Jost} of the form 
\be
\label{eq:PDE1}
\begin{array}{rcll}
\frac{\partial u}{\partial t}-\mathcal{D}\Delta u&=& F(x,t,u;\mu) & \text{for $x\in \Omega$, $t> 0$}\\
u(x,t)&=&g(x,t) & \text{for $x\in \partial\Omega$, $t>0$}\\
u(x,0)&=&h(x) & \text{for $x\in \partial\Omega$}\\
\end{array}
\ee  
for $g:\Omega\times \R^+\ra \R^n$, $h:\Omega \ra \R^n$, $\Delta$ is the Laplacian and $\mathcal{D}$ is a diagonal matrix called the diffusion matrix with non-negative elements. We shall restrict our presentation of generalized models here to \eqref{eq:PDE1} but remark that the ideas naturally extend to many other equations within the class \eqref{eq:PDE}. Assume that $F$ can again be decomposed into gain and loss terms 
\benn
F_i(u,x,t)=\sum_{k=1}^{K_i}a_{i,k}F_{i,k}(u,x,t).
\eenn
where we omit the parameters $\mu\in\R^p$ again for notational convenience. Note that this assumption is very natural in the context of reaction-diffusion systems since $F$ represents the reaction terms. For example, if we think of a chemical reaction then gain terms would be terms that increase a concentration while loss terms decrease it. Suppose \eqref{eq:PDE1} there exists a constant function $u^*$ such that 
\benn
u^*(x,t)=(u_1^*,\ldots,u_n^*),\qquad,  \qquad u^*_i\in\R, \quad u^*_i\neq 0,
\eenn
for all $t>0$. Then $u^*$ is a space-time homogeneous equilibrium solution to \eqref{eq:PDE1}. We can apply the normalizing coordinate change
\benn
v_i:=\frac{u_i}{u_i^*}, \qquad \text{for $i\in\{1,2,\ldots,n\}$.}
\eenn 
Then the reaction-diffusion system \eqref{eq:PDE1} transforms to
\be
\label{eq:PDE2}
\begin{array}{rcll}
\frac{\partial v_i}{\partial t}-\alpha_i\tilde{\mathcal{D}}_{ii}\Delta v_i&=& 
\frac{F_{i}(x,t,v_1u_1^*,\ldots,v_nu_n^*)}{u_i^*}=:\alpha_i\sum_{k=1}^{K_i}a_{i,k}
\beta_{i,k} f_{i,k}(x,t,v), & \text{$x\in \Omega$, $t> 0$}\\
v_i(x,t)&=&\frac{g_i(x,t)}{u_i^*}=:g_i^*(x,t), & \text{$x\in \partial\Omega$, $t>0$}\\
v_i(x,0)&=&\frac{h_i(x)}{u_i^*}=:h_i^*(x), & \text{$x\in \partial\Omega$}\\
\end{array}
\ee 
where the usual generalized modeling definitions as in \eqref{eq:ODE_scaled2} are used and the diffusion matrix has been rescaled $\tilde{\mathcal{D}}_{ii}=\mathcal{D}_{ii}/\alpha_i$. Then \eqref{eq:PDE2} has a space-time homogeneous equilibrium at $v=(1,1,\ldots,1)=:1$. Now all linearization techniques at $v=1$ for \eqref{eq:PDE2} can make use of the interpretation of scale parameters and elasticities as before. As a typical example one can consider the Turing-Hopf mechanism \cite{Jost}. Let $J(1)$ denote the Jacobian as in \eqref{eq:Jac_gen}. If $J(1)$ has only eigenvalues $\lambda=\lambda(k)$ with $\Re(\lambda)<0$ then the equilibrium $v=1$ is stable as a solution to $v'=J(1)v$. The Turing-Hopf bifurcation requires that some eigenvalue of the Laplacian has $\Re(\lambda(k))>0$ for some $k\geq 1$ which results in a spatial instability and associated pattern formation. Explicit conditions for this scenario to occur can then be derived in terms of the generalized parameters and the rescaled diffusion coefficients \cite{BaurmannGrossFeudel}.

\subsection{Delay Differential Equations}

Generalized models for delay equations have first been considered in the context of coupled oscillators in \cite{HoefenerSethiaGross}. We are going to consider a slightly more general delay differential equation (DDE) with constant delays 
$t_l$, $l\in\{1,2,\ldots,m\}$, given by
\be
\label{eq:DDE}
X'=F(X(t),X(t-\tau_1),X(t-\tau_2),\ldots,X(t-\tau_m);\mu)
\ee
where $F:\R^{(m+1)n}\times \R^p\ra \R^n$ and we abbreviate $X(t-\tau_l)$ as $X^{\tau_l}$. As for PDEs, we remark that \eqref{eq:DDE} only presents a subclass of DDEs as it is non-neutral with constant delays  \cite{HaleLunel} and we expect generalized modeling to apply for many more DDEs than just \eqref{eq:DDE}. Suppose there exists an equilibrium point $X=X^*$ so that
\benn
F(X^*,X^*,\ldots,X^*;\mu)=0.
\eenn
Applying the usual generalized modeling procedure we end up with  
\be
\label{eq:scaled_DDE}
x_i'=\alpha_i\left(\sum_{k=1}^{K_i}a_{i,k}\beta_{i,k}f_{i,k}(x)\right)
\ee
where the scale parameters are defined as usual and 
\benn 
f_{i,k}(x)=\frac{F_{i,k}(X^*\cdot x,X^*\cdot x^{\tau_1},\ldots,X^*\cdot x^{\tau_n})}{X_i^*}.
\eenn
The linearized equation for \eqref{eq:scaled_DDE} around $x=1$ is
\be
\label{eq:DDE_lin}
v'=\sum_{l=0}^mA_lv^{\tau_l}
\ee
where the $n\times n$ matrices $A_l$ consists of the rows
\benn
\alpha_i\left(\sum_{k=1}^{K_i}a_{i,k}\beta_{i,k}\sum_{l=0}^m (D_lf_{i,k})(1)\right)
\eenn
and where $D_l$ denotes the total derivative of $f_{i,k}$ with respect to the $l$-th argument; note that we have employed the convention $\tau_0=0$ so that $l=0$ denotes the first argument. The characteristic equation associated to \eqref{eq:DDE_lin} is obtained by assuming an exponential solution of the form $e^{\lambda t}$ and is given by
\be
\label{eq:DDE_char}
\det\left(\lambda \text{Id}-\sum_{l=0}^m e^{-\lambda \tau_l} A_l \right)=0.
\ee
In contrast to ODEs, we see that \eqref{eq:DDE_char} is a transcendental equation which can have an infinite number of solutions $\lambda$. It is known that if $\Re(\lambda)<0$ for every solution then the solution is stable. Bifurcation analysis for \eqref{eq:DDE_char} can then be carried out using numerical \cite{EngelborghsLuzyaninaSamaey} or analytical \cite{Erneux} methods. The generalized parameters are used as bifurcation parameters in this context.

\subsection{Stochastic Differential Equations}
\label{ssec:SDE}

Consider a system of stochastic differential equations (SDEs) \cite{Oksendal}
\be
\label{eq:SDE}
dX=F(X;\mu)dt+G(X;\mu)dW
\ee
where $W=W(t)=(W_1(t),W_2(t),\ldots,W_k(t))^T$ is a $k$-dimensional Brownian motion, $F:\R^n\ra \R^n$, $G:\R^n\ra \R^{n\times k}$ is a matrix-valued function. The normalizing coordinate change \eqref{eq:rescale} can be applied to \eqref{eq:SDE} for a deterministic equilibrium $F(X^*)=0$ with $X_i^*\neq 0$ for all $i\in\{1,\ldots,n\}$. Since the coordinate change is linear the It\^{o} formula \cite{Gardiner,Kallenberg} reduces to the standard chain rule and we get 
\benn
dx_i=\frac{F_i(X_1x_1,\ldots,X_nx_n;\mu)}{X_i^*}dt+
\frac{1}{X_i^*}G_i(X_1x_1,\ldots,X_nx_n;\mu)dW, 
\quad \text{for $i\in\{1,2,\ldots,n\}$}
\eenn
where $G_i$ indicates the $i$-th row of $G$. The scale and exponent parameters for the deterministic drift terms $F_i(X;\mu)/X_i^*$ can be defined as in Section \ref{sec:gm}. We can also consider a normalized matrix-valued function for the diffusion term where each row is given by 
\benn
g_i(x;\mu):=\frac{G_i(X_1x_1,\ldots,X_nx_n;\mu)}{X_i^*}.
\eenn
Then we can still formally linearize \eqref{eq:SDE} at $x=1$ and obtain to lowest order
\be
\label{eq:SDE1}
d\xi=\left(\left.\frac{\partial }{\partial x_j} 
\sum_{k=1}^{K_i} f_{i,k}(x;\mu)\right|_{x=1}\right)_{i,j} \xi dt
+g(1;\mu)dW
\ee
where $\xi=\xi(t)$ now solves an SDE with linear drift term and constant diffusion. There are multiple possibilities on how to develop a ``stochastic bifurcation theory'' \cite{HorsthemkeLefever,ArnoldSDE}. Therefore we shall not discuss generalized modeling for SDEs in any more detail. However, we expect generalized models for SDEs to work for bifurcation and stability analysis as well.

\subsection{Nonlocal Generalized Models}
\label{ssec:nonlocal}

So far, we have presented a systematic approach to generalized models to analyze local dynamics. It is a natural question to ask whether the approach can be used to analyze invariant sets beyond equilibria such as periodic orbits, homoclinic orbits or tori.\\

A first important point to notice in this regard is that generalized models already provide a way to locate certain global orbits via local bifurcations. For example, it is well-known \cite{Kuznetsov}, under certain assumptions on normal form coefficients, that
\begin{itemize}
 \item Codimension one Hopf bifurcations imply the existence of a periodic orbit,
 \item Codimension two Bogdanov-Takens points imply the existence of a homoclinic orbit,
 \item Codimension two Hopf-Hopf bifurcations imply the existence of an invariant torus.
\end{itemize}

In fact, under suitable conditions on equilibrium point for a saddle-focus homoclinic orbit or via the break-up of tori near Hopf-Hopf bifurcation one can infer the existence of chaotic invariant sets which are obviously global phenomena. This approach has been used for generalized models in several applications \cite{GrossEbenhoehFeudel,GrossFeudel1,Stiefs,StiefsVenturinoFeudel,ZumsandeGross,ZumsandeStiefsSiegmundGross}; see also Figure \ref{fig:Stiefs}.\\

However, this approach cannot capture global bifurcations such as saddle-nodes of periodic orbits. Recent work of Kuehn and Gross \cite{KuehnGross} shows how to extend generalized modeling to periodic orbits for the predator-prey system \eqref{eq:gm_local}. The main problem is that the generalized parameter $\beta_{i,k}$ and $f_{i,k,x_j}$ become time-dependent functions $\beta_{i,k}(t)$ and $f_{i,k,x_j}(t)$. This changes the algebraic structure of the generalized model and introduces a so-called moduli flow constraint on the scale functions $\beta_{i,k}(t)$. The detailed description of this approach is beyond the scope of this paper and we refer the interested reader to \cite{KuehnGross}.

\bibliographystyle{plain}
\bibliography{../my_refs}

\end{document}